\documentclass[a4paper,10pt,leqno,english]{amsart}
\title[Mountain  Pass  Theorem]
{Mountain  Pass  Theorem  with infinite  discrete  symmetry}
\author{No\'{e} B\'{a}rcenas }
\address{Centro de Ciencias Matem\'aticas. UNAM \\ Ap.Postal 61-3 Xangari. Morelia, Michoac\'an M\'EXICO 58089}
\email{barcenas@matmor.unam.mx}
\urladdr{http://www.matmor.unam.mx/~barcenas}


\usepackage{hyperref}
\usepackage{pdfsync}
\usepackage{calc}
\usepackage{enumerate,amssymb}
\usepackage[arrow,curve,matrix,tips,2cell]{xy}
  \SelectTips{eu}{10} \UseTips
  \UseAllTwocells

\DeclareMathAlphabet\EuR{U}{eur}{m}{n}
\SetMathAlphabet\EuR{bold}{U}{eur}{b}{n}

\makeindex             

\theoremstyle{plain}
\newtheorem{theorem}{Theorem}[section]
\newtheorem{lemma}[theorem]{Lemma}
\newtheorem{proposition}[theorem]{Proposition}
\newtheorem{corollary}[theorem]{Corollary}

\newtheorem{problem}[theorem]{Problem}

\theoremstyle{definition}
\newtheorem{definition}[theorem]{Definition}

\newtheorem{condition}[theorem]{Condition}
\newtheorem{remark}[theorem]{Remark}

{\catcode`@=11\global\let\c@equation=\c@theorem}

\newcommand{\comsquare}[8]                   
{\begin{CD}
#1 @>#2>> #3\\
@V{#4}VV @V{#5}VV\\
#6 @>#7>> #8
\end{CD}
}

\newcommand{\xycomsquare}[8]                   
{\xymatrix
{#1 \ar[r]^{#2} \ar[d]^{#4} &
#3 \ar[d]^{#5}  \\
#6\ar[r]^{#7} &
#8
}
}

\newcommand{\calfin}{\mathcal{FIN}}

\newcommand{\IZ}{{\mathbb Z}}

\newcommand{\curs}{\EuR}

\newcommand{\SPHB}{\curs{SPHB}}

\newcommand{\pt}{\{\bullet\}}

\newcommand{\EGF}[2]{E_{#2}(#1)}                   
\newcommand{\eub}[1]{\underline{E}#1}              

\newcommand{\higherlim}[3]{{\setbox1=\hbox{\rm lim}
        \setbox2=\hbox to \wd1{\leftarrowfill} \ht2=0pt \dp2=-1pt
        \mathop{\vtop{\baselineskip=5pt\box1\box2}}
        _{#1}}^{#2}#3}

\newcommand{\version}[1]                       
{\begin{center} last edited on #1\\
last compiled on \today\\
name of texfile: \jobname
\end{center}
}

\newcounter{commentcounter}

\begin{document}

\typeout{------------------------------------ Abstract ----------------------------------------}

\begin{abstract}
We  extend  an equivariant   Mountain  Pass  Theorem, due  to Bartsch, Clapp and  Puppe for  compact  Lie  groups  to the  setting  of  infinite  discrete groups  satisfying  a maximality condition on  their  finite  subgroups.
 \end{abstract}

\maketitle
 \typeout{-------------------------------   Section 1: Introduction --------------------------------}
\setcounter{section}{1}   

Symmetries  play a  fundamental  role in  the  analysis  of  critical  points  and  sets  of   functionals \cite{ambrosettirabinowitz}, \cite{marzantowiczkrawkewicz}, \cite{mountainpassthmclapp}.   The  development  of Equivariant Algebraic  Topology,  particularly  Equivariant Homotopy  Theory, has given  a  number  of  tools  to  conclude  the  existence  of  critical  points   in  problems  which  are  invariant under  the  action  of  a  compact Lie  group, as investigated  in \cite{bartschhabilitation}.

In  this  work  we  discuss extensions  of  methods  of  Equivariant Algebraic   Topology   to  the  setting  of  actions  of  infinite  groups. The main  result  of  this  note is   the  modification   of   a  result  by  Bartsch,  Clapp and  Puppe originally proved for actions  of compact  Lie  groups,   to  infinite  discrete   groups  with  appropriate  families  of  finite  subgroups  inside  them.
 
\begin{theorem}[Mountain Pass  Theorem]\label{theorem mountainpass}
 Let  $G$  be  an  infinite  discrete  group acting  by  bounded  linear operators on a  real  Banach  space  $E$  of  infinite  dimension. Suppose  that  $G$  satisfies  the  maximality  condition  \ref{condition m}  and  that the  linear  action is  proper  outside $0$.  Let  $\phi:E \to \mathbb{R}$  be  a   $G$-invariant  functional of  class $\mathcal{C} ^{2-}$. For  any   value $a\in \mathbb{R}$,   define   the sublevel set $\phi^{a}=\{x \in  E\mid \phi(x)\leq a\}$ and  the  critical  set  $K=\cup_{c\in \mathbb{R}}K_{c}$, where $K_{c} $  is  the    critical  set  at  level $c$, $K_{c}=\{u\mid   \|   \phi^{'}(u) \|=0  \; \phi(u)=c\}$. Suppose that  
 \begin{itemize}
 \item{$\phi(0)\leq a$ and  there exists  a   linear  subspace $\hat{E}\subset E$  of  finite  codimension such that $\hat{E}\cap \phi^{a}$ is  the  disjoint  union of two closed subspaces,   one  of  which is  bounded  and contains  $0$. }
 
\item{The  functional $\phi$ satisfies   the  Orbitwise Palais-Smale  condition \ref{condition OPS}. }
\item{The  group $G$  satisfies  the  maximal finite  subgroups  condition  \ref{condition m}. } 
\end{itemize}
 
Then,  the  equivariant Lusternik-Schnirelmann category  of  $E$ relative  to  $\phi^{a}$,  $G-cat(E, \phi^{a})$  is  infinite. If  moreover,  the  critical  sets $K_{c}$  are  cocompact under  the  group   action, meaning  that  the  quotient spaces $G\diagdown K_{c} $  are  compact, then    $\phi(K) $  is  unbounded  above.  
\end{theorem}

Recall  that    given a  natural  number $r$,  the  class  $\mathcal{C}^{r-}$  denotes  the  class  of  functions  whose  derivatives  up  to  order  $r$  exist  and  are  locally  Lipschitz.

Condition  \ref{condition m}  restricts   maximal  finite  subgroups  and  their  conjugacy  relations. 
\begin{condition}\label{condition m}
Let $G$  be  a  discrete  group  and  $\mathcal{MAX}$  be a  subset  of  finite  subgroups.  $G$ satisfies  the  maximality condition  if  

 \begin{itemize}

  \item{There  exists  a  prime number   $p$ such  that  every  nontrivial  finite  subgroup  is  contained  in  a  unique  maximal  $p$-group $M  \in \mathcal{MAX}$. }
  \item{$M\in  \mathcal{MAX}\Longrightarrow N_{G}(M)=M $,  where $N_{G}(M)$ denotes  the  normalizer  of  $M$ in $G$. }
 \end{itemize}

Notice  that  in particular,  the finite  subgroups  of  $G$  are  all   finite $p$- groups. 
 
\end{condition}

These  conditions  are  satisfied  in  several    cases.  Among  them: 
\begin{enumerate}
 \item{Extensions  $1\to\mathbb{Z}^ {n}\to G\to K\to 1$  by  a  finite $p$-group  given  by  a  representation $K\to Gl_{n}(\mathbb{Z})$  acting  freely  outside  from  the  origin \cite{lueckstamm}, Lemma  6.3.  }
\item{Fuchsian groups, more  generally   NEC (non-euclidean  crys\-ta\-llo\-gra\-phic \\ groups) for  which  the  isotropy  consists  only  of  $p$-groups. \cite{lueckstamm}. }
\item{One  relator  groups $G=\langle q_{i}\mid  r \rangle$ for  which  the family of   finite  subgroups  consists  of   p-groups. See  \cite{lyndonschupp}, Propositions  5.17,  5.18, 5.19. in pages  107  and  108. }
\end{enumerate}

The  Orbitwise Palais-Smale  condition  was  formulated  by  Ayala-Lasheras-Quintero in  \cite{ayalalasherasquintero}  for complete riemannian  manifolds with  a  proper  action  of  a   Lie  Group. For  our  purposes,  the  following  notion  is  more  adequate. 
\begin{condition}
\label{condition OPS}
Let  $G$  be  a  discrete  group. Let M  be  a $\mathcal{C}^{2-}$  Hilbert  manifold  with a $G$-action  by  $\mathcal{C}^{1-}$ diffeomorphisms  which  is  proper. Assume  that  $M$ has  a  $G$-invariant  $\mathcal{C}^{1-}$ Riemannian  Metric. 
The  $G$-invariant  functional $\Phi$ of  class  $\mathcal{C}^{2-}$ satisfies  the  orbitwise  Palais-Smale  condition  if   given a  sequence $\{ x_n\} \subset  M$   such  that $\mid f(x_n)\mid$  is  bounded  and $\nabla \Phi( x_n)$  converges  to  $0$, then  the  sequence of  orbits  $Gx_n$  contains  a   convergent  subsequence  in  the  orbit  space $M/G$.       

\end{condition}

This  paper  is  organized  as  follows: in the  second  section ,  the  usual   facts   concerning  the  relation  between critical  points,   Lusternik-Schnirelmann category and equivariant  deformation  theorems   are  stated, being   modified  slightly  from  \cite{ayalalasherasquintero} and  \cite{clapppuppelusternik}. 

In  the  third  section,  we introduce  the  notion  of  Universal  Proper Length related  to  a  family  of  subgroups. 

In the  third section,   we    use  some  algebraic  properties  of   the  classifying  space  for  proper  actions  of  groups  with  an  appropriate  family  of  maximal   finite  subgroups  in order  to  conclude the  unboundedness  of  critical  values.

This is  done  in the  fourth section  adapting  a   construction  of  elements  in  the  Burnside  Ring  of  a  finite  group,  originally  due  to  Bartsch, Clapp  and  Puppe  \cite{mountainpassthmclapp}  to  the  infinite  group  setting,  using   the  Atiyah-Hirzebruch  spectral  sequence, as  well  as  a  version  of  the Segal  Conjecture  for  families  of  finite  groups  inside discrete  groups \cite{luecksegal}, \cite{barcenastesis}.

This  work  was  financially  supported by   the Hausdorff  Center  for  Mathematics at  the  University  of  Bonn,  Wolfgang  L\"uck's  Leibnizpreis  and a  CONACYT  postdoctoral  fellowship.   The  author  thanks  the  comments  of  an  anonymous  referee.

\section{Proper  Lusternik-Schnirelmann Category and  Critical  Points}

\typeout{--------------------------------- Section  1: Proper  Lusternik-Schnirelmann Category, Universal  Length    and  Critical  Points-----------}

The  notion  of  a  proper  $G$-space  provides   an  adequate  setting   for  the  study   of  non-compact  transformation  groups. 

\begin{definition}
Let $G$  be  a second  countable, Hausdorff locally  compact  group. Let  $X$  be  a  second  countable, locally Hausdorff space. Recall  that  a  $G$-action  is  proper  if  the  map 
$$\underset{\overset{\theta_{X}} {(g,x)\mapsto (x,gx)} } { G\times  X\to X\times X}$$ 
is  proper.
\end{definition}

 Ayala-Lasheras-Quintero \cite{ayalalasherasquintero}   introduced the   notion of  equivariant  Lus\-ter\-nik-Schni\-rel\-man category for  proper  actions  of  Lie  Groups, extending  previous  work  by  Marzantowicz \cite{Marzantowiczlusternik} for  compact Lie groups.  

\begin{definition}
Let $X^{'}\subset X$ be paracompact  proper  $G$-spaces. The relative  $G$-category  of  $(X,X^{'})$, denoted  by $G$-$cat(X, X^{'})$ is  the   smallest  number  $k$  such  that  $X$ can  be  covered  by $k+1$ open $G$-subsets $X_{0}, X_{1}, \ldots, X_{k}$ with the  following  properties: 
\begin{itemize}
\item{$X^{'}\subset X_{0}$ and   there  is  a  homotopy $H:(X_{0},X^{'})\times I \to (X_{0}, X^{'})$ starting   with  the  inclusion  and   $H(x,1)\in X^{'}$.}
\item{For  every $i\in \{1,\ldots,k\}$ there  exist  $G$-maps $\alpha_{i}:X_{i}\to A_{i}$  and  $\beta_{i}:A_{i}\to Y$  with $A_{i}$  a $G$-orbit $G/H_{i}$  such  that   the  restriction  of $f$  to $X_{i}$  is  the   is  $G$-homotopic  to  the  composition $\beta_{i} \circ \alpha_{i}$}
\end{itemize}
If  no  such a  number  exists,  then  we  write $G$-$cat(X,  X^{'})=\infty$.

\end{definition}

The   Lusternik-Schnirelman  Method  can  be  extended   to  functionals which are   invariant  under   proper  actions.

\begin{lemma} [Equivariant  Deformation] \label{lemma deformation}
Let  $G$  be  a discrete group  acting  properly   on  a  Hilbert manifold  of  class $\mathcal{C}^{2-}$.    Let   $\Phi: X\to \mathbb{R}$  be  a  $G$-  invariant  $\mathcal{C}^{2-}$-functional,   $c\in K_c= \{ x \in X\mid  \Phi^{'}(x)=0 \, \Phi(x)= c\} $. 
For  every $c>a$, every  $0<\delta< c-a$  and  every  $G$-neighborhood  $U$ of  $K_{c}$,  there  is   an  $\epsilon>0 $  and  a  homotopy $\eta: \Phi^{c+\epsilon}\times  I  \to \Phi^{c-\epsilon} $  which  is  the  identity on  $\Phi^{c-\delta} \times  I $. 

\end{lemma}

 \begin{proof}
The  Gradient  field  $-\nabla \Phi$ is  locally  Lipschitz   by  assumption. The  usual  deformation method  \cite{rabinowitz}  works  $G$-equivariantly. See  \cite{ayalalasherasquintero},  lemma  5.4  in page  1130.   

 \end{proof}

\begin{proposition}
Let $M$  be  a  paracompact   Hilbert, $\mathcal{C}^{2-}$- manifold. Assume  that  the  discrete  group  $G$  acts  properly  by  $\mathcal{C}^{1-}$ maps  on $M$. 
 Let $\phi: M  \to  \mathbb{R}$  be  a  $G$-invariant $C^{2-}$-function  satisfying  the   deformation property with  respect  to  neighbourhoods of critical sets, as   in  lemma  \ref{lemma deformation}.  Suppose  that $\Phi$  satisfies the  Orbitwise  Palais-Smale  condition \ref{condition OPS}. 
\begin{itemize}
 \item{ If  the  function  is  bounded  below, then the  number  of critical  points  of  $\phi$ with   values $>a$ in $M$ is  at least  $G$-$cat( M, \phi^{a})$. }
 \item{If  $G$-$cat(M, \phi^{a})$ is  greater  than  the  number  of  critical  values  of  $\phi$ above  $a$, then  there  is  at  least  one  $c>a$  such  that   the critical  set  $K_{c}$ has  positive  covering  dimension.  In particular  $\phi$ has  infinitely  many  critical  orbits with  values  above  $a$. }
 \item{If $G$-$cat(M, K)=\infty$, then  $\phi$ has an  unbounded  sequence   of  critical  values.  }
 \end{itemize}
\end{proposition}
\begin{proof}
 The  proofs  given  in  \cite{clapppuppelusternik},  Theorem  2.3  and  Corollary  2.4 , pages  606  and  607, and  \cite{clapppuppe}, Theorem 1.1   extend  to the proper  setting.  The point  is  that  the equivariant  Lusternik-Schnirelmann  Category  for  proper  spaces satisfies  subadditivity,  deformation monotonicity,  and  continuity (Proposition 2.3  in \cite{ayalalasherasquintero} in the  absolute  case,  and  the   obvious  modification extends  to the  relative  category).   
\end{proof}

\section{Universal  Cohomology  Length }

We  discuss  now  cohomology  length in the  context  of  equivariant  cohomology  theories. We  use  for  this   the  notion  of  a  classifying  space  for a  family  of  subgroups.

\begin{definition}
Recall  that  a  $G$-CW  complex structure  on  the  pair $(X,A)$  consists  of a  filtration of  the $G$-space $X=\cup_{-1\leq n } X_{n}$, $X_{-1}=\emptyset$,$X_{0}=A$  and  for   which  every  space  $X_{n}$ is inductively  obtained  from  the  previous  one   by  attaching  cells  in pushout  diagrams  of  the  form
$$\xymatrix{\coprod_{i} S^{n-1}\times G/H_{i} \ar[r] \ar[d] & X_{n-1} \ar[d] \\ \coprod_{i}D^{n}\times G/H_{i} \ar[r]& X_{n}}$$   
We  say  that a  proper  $G$-CW complex  is  finite  if  it  consists  of  a  finite  number  of  cells  $G/H\times  D^{n}$. 
\end{definition}

\begin{definition}
Let  $G$  be  a  discrete  group. A  metrizable  proper  $G$-Space  $X$  is  an  Absolute  Neighbourhood  retract  if   every $G$- map  $Z\to X$  from a  closed  subspace  $Z$    of  a  metrizable  $G$-space $Y$  into  $X$ has  an  equivariant extension $U\to X$  to  a $G$-invariant  neighbourhood  $U$ of $Z$ in  $Y$. 

\end{definition}
It  is  proved in \cite{antonyanelfving} , Theorem  1.1  that   proper  $G$-ANR  are   $G$-homotopy  equivalent  to $G$-CW complexes when $G$  is  a  locally compact Hausdorff group.  

 We  recall the  notion  of  the  classifying  space  for  a  family  of subgroups. 

\begin{definition}
 Let  $\mathcal{F}$  be  a   collection  of  subgroups  in  a  discrete  group  $G$which is  closed  under  conjugation  and  intersection. A model  for  the  classifying  space for  the  family  $\mathcal{F}$  is  a $G$-CW  complex $X$  satisfying  
\begin{itemize}
  \item{All  isotropy  groups of  $X$  lie  in $\mathcal{F}$ .}
\item{For  any $G$-CW  complex $Y$ with  isotropy  in $\mathcal{F}$, there  exists up  to  $G$-homotopy   a  unique  $G$-equivariant  map $f:Y\to X$. }
\end{itemize}
\end{definition}
A  model  for  the classifying  space  of  the  family  $\mathcal{F}$  will be  usually  denoted  by  $ \EGF{G}{\mathcal{F}}$ . 

Particularly  relevant   is  the classifying  space  for  proper actions,  the  classifying  space  for   the  family $\calfin$  of  finite  subgroups, denoted  by  $\eub{G}$.

The  classifying  space  for  proper  actions  always  exists, is  u\-ni\-que  up  to $G$-homotopy  and  ad\-mits  several  models. The  following  list   includes  some  examples. We  remit  to \cite{lueckclassifying} for  further  discussion. 
 \begin{itemize}
\item{If  $G$  is  a  compact  group,  then  the  singleton  space  is  a  model  for  $\eub G$. }
\item{Let $G$ be  a  group  acting properly  and  co-compactly on  a     ${\rm CAT}(0)$  space  $X$, in the  sense  of  \cite{bridsonhaefliger}. Then  $X$ is a model for  $\eub{G}$. }
\item{Let  $G$  be  a  Coxeter  group. The  Davis Complex  is  a model for   $\eub{G}$.}
\item{Let  $G$ be  a   mapping  class  group  of  an orientable  surface. The  Teichm\"uller  Space  is a  model  for $\eub G$. } 

\end{itemize}

The  spaces apearing  in  applications  in  analysis are  not  always  $G$-CW  complexes.  They satisfy  more  often   numerability  conditions.

\begin{definition}
Let $\mathcal{F}$ be  family of  closed  subgroups closed  under  conjugation  and  intersection     inside  the  locally  compact  second  countable  Hausdorff  group $G$. A $G$-space  $X$  is   said  to be  an $\mathcal{F}$-numerable space  if   there exists  an  open  covering $\{ U_{i},\mid  i\in I \}$ by  $G$-subspaces  such that  there  is  for each $i\in I$  a $G$-map $U_{i}\to G/G_{i}$  for  some  $G_{i}\in \mathcal{F}$ and  there  is  a  locally  finite  partition  of  unity $\{e_{i \mid  i \in I }\}$ subordinate  to  $\{U_{i}\}$ by  $G$-invariant  functions.  Notice  that  we do  not  require  that  the isotropy  groups   of  $X$  lie  in $\mathcal{F}$.  

\end{definition}

The  Slice Theorem  2.3.3, in  page  313  of  \cite{palais} implies  that completely  regular  spaces  carrying proper  actions  of Lie  groups  are  precisely numerable  spaces   with  respect  to the  family  of  compact subgroups  for  which,   in addition,   the  isotropy  groups  of  points  are all  compact  subgroups. 

Specializing  to  Lie  groups  acting  properly on  $G$-CW  complexes, the conditions boil  down  to the  fact  that  all  stabilizers  are  compact, see \cite{luecktransformation},  Theorem  1.23.  In particular  for  a cellular action  of  a  discrete  group  $G$  on  a  $G$-CW  complex,  a  proper  action  reduces   to  the  finiteness of  all  stabilizer  groups.  Notice  that any  (continuous) action  of  a  compact  Lie  group  or  a  finite  group  on  a   locally compact, Hausdorff space  is  proper.

The  following  version  of  the  classifying  space  for  a  family   extends  the notion  to $\mathcal{F}$-numerable  spaces.

\begin {definition} [Numerable Version  for  the  Classifying  space of a  family]

Let $\mathcal{F}$ be  a  family  of  subgroups.  A  model  $J_{\mathcal{F}}(G)$   for  the classifying  numerable  $G$-space  for  the  family  $\mathcal{F}$  is  a  $G$-space  which  has  the  following properties: \begin{itemize}
\item{$J_{\mathcal{F}}(G)$  is  $\mathcal{F}$-numerable }
\item{For  any  $\mathcal{F}$-numerable  space  $X$  there  is  up to  $G$-homotopy  precisely one  map  $X\to  J_{\mathcal{F}}(G)$.}
\end{itemize}
\end{definition}

\begin{remark}
There  exists  up  to  $G$-homotopy  a unique $G$-equivariant  map $\eub{G}\to  J_{\mathcal{F}}(G)$.  This  map is  proved  to  be  a  $G$-homotopy  equivalence  for a   discrete   group   in  Theorem  3.7, part  ii  of  \cite{lueckclassifying}. 
\end{remark}

Recall  the  notion of  an Equivariant  Cohomology  Theory, \cite{lueckeqcohomological}.

\begin{definition}
 Let  $G$  be  a  group  and  fix  an  associative  ring  with  unit $R$. A $G$-Cohomology  Theory  with  values  in  $R$-modules is  a  collection  of  contravariant  functors $\mathcal{H}^{n}_{G}$ indexed  by  the  integer  numbers $\mathbb{Z}$ from  the  category  of  $G$-$CW$  pairs together  with  natural  transformations $\partial^{n}_{G}: \mathcal{H}^{n}_{G}(A):=\mathcal{H}^{n}_{G}(A, \emptyset)\to \mathcal{H}^{n+1}_{G}(X,A)  $, such  that  the  following axioms  are  satisfied: 

\begin{enumerate}
 \item{If  $f_{0}$  and  $f_{1}$  are  $G$-homotopic  maps $(X,A)\to (Y,B)$ of   $G$-CW  pairs, then  $\mathcal{H}^{n}_{G}(f_{0})=\mathcal{H}^{n}_{G}(f_{1})$ for  all n. }
 \item{Given a  pair $(X,A)$  of  $G$-$CW$  complexes, there  is  a  long  exact  sequence 
\begin{multline*}
$$ \ldots \overset{\mathcal{H}^{n-1}_{G}(i)} {\rightarrow}   \mathcal{H}^{n-1}_{G}(A)   \overset{\partial^{n-1}_{G}} {\rightarrow}   \mathcal{H}^{n}_{G}(X,A) \overset{\mathcal{H}^{n}_{G}(j)}{\rightarrow} \mathcal{H}^{n}_{G}(X) \\ \overset{\mathcal{H}^{n}_{G}(i)} {\rightarrow} \mathcal{H}^{n}_{G}(A) \overset{\partial^{n}_{G}} {\rightarrow} \mathcal{H}^{n+1}_{G}(X,A)  \overset{\mathcal{H}_{n+1}(j)}{\rightarrow} \ldots $$
\end{multline*}

 where $i:A\to X$ and  $j:X\to (X,A)$ are  the  inclusions. }

\item{Let  $(X,A)$ be  a  $G$-$CW$ pair  and  $f: A\to B$  be  a   cellular  map. The  canonical  map $(F,f): (X,A)\to (X\cup_{f} B, B)$ induces  an  isomorphism 
$$ \mathcal{H}^{n}_{G}(X\cup_{f}B, B) \overset{\cong}{\to} \mathcal{H}^{n}_{G}(X,A)$$ }

\item{ Let $\{ X_{i}\mid i\in \mathcal{I} \}$  be  a  family  of  $G$-$CW$-complexes and  denote  by $j_{i}: X_{i}\to  \coprod_{i\in \mathcal{I}} X_{i}$ the  inclusion  map. Then  the  map  
$$\Pi_{i\in \mathcal{I}}\mathcal{H}^{n}_{G}(j_{i}):  \mathcal{H}^{n}_{G}(\coprod_{i}X_{i})   \overset{\cong}{\to}  \Pi_{i\in \mathcal{I}}\mathcal{H}^{n}_{G}(X_{i})$$
is  bijective  for  each  $n\in \mathbb{Z}$. }

\end{enumerate}
A $G$-Cohomology  Theory  is  said  to  have  a  multiplicative  structure  if   there  exist  natural, graded  commutative $\cup$- products  

$$\mathcal{H}^{n}_{G}(X,A)\otimes \mathcal{H}^{m}_{G}(X,A) \to \mathcal{H}^{n+m}_{G}(X,A)$$

Let $\alpha:H\to G$ be  a  group  homomorphism and  $X$ be  a $H$-CW  complex. The  induced space ${\rm ind}_{\alpha}X,$ is defined  to be  the  $G$-CW complex  defined  as  the  quotient space $G\times X $ by  the  right  $H$-action  given  by $(g,x)\cdot h =( g\alpha(h),h^{-1}x)$. 

An Equivariant  Cohomology  Theory consists  of a  family of $G$-Cohomology Theories  $\mathcal{H}^{*}_{G}$  together  with
an induction structure determined  by  graded ring  homomorphisms 

$$ \mathcal{H}^{n}_{G}({\rm ind}_{\alpha}(X,A))\to   \mathcal{H}_{H}^{n}(X,A) $$ 
which  are  isomorphisms  for  group  homomorphisms $\alpha: H\to G$  whose kernel acts freely
on $X$ satisfying  the  following  conditions: 
\begin{enumerate}
 \item{For  any  $n$, $\partial^{n}_{H}\circ {\rm ind}_{\alpha}= {\rm ind}_{\alpha}\circ \partial^{n}_{G}$.}
 \item{For any  group  homomorphism $\beta: G\to K$ such  that $\ker \beta\circ \alpha$ acts freely on  $X$,  one  has 
$${\rm  ind}_{\alpha\circ \beta}= \mathcal{H}^{n}_{K}(f_{1}\circ {\rm ind}_{\beta}\circ{\rm ind}_{\alpha}): \mathcal{H}_{K}^{n}({\rm ind}_{\beta\circ \alpha}(X,A))\to  \mathcal{H}^{n}_{H}(X,A)$$
where $f_{1}: {\rm ind}_{\beta}{\rm ind}_{\alpha}\to {\rm ind}_{\beta\circ\alpha}$ is  the  canonical $G$-homeomorphism.}                                      

\item{For  any $n\in \mathbb{Z}$, any  $g\in G$ , the  homomorphism 

$${\rm ind}_{c_(g):G\to G}: \mathcal{H}^{n}_{G}(\rm ind)_{c(g):G\to G}(X,A))  \to \mathcal{H}^{n}_{G}(X,A) $$

agrees  with  the  map  $\mathcal{H}^{n}_{G}(f_{2})$, where  $f_{2}: (X,A)\to {\rm ind}_{c(g):G\to G}$ sends $x$  to  $(1,g^{-1}x)$ and $c(g)$  is  the  conjugation  isomorphism in $G$.}
\end{enumerate}
\end{definition}
\begin{remark}[Extensions  of $G$-Cohomology  theories to  more  general  spaces]
Let $\mathcal{H}_{G}^{*}$ be  a  $G$-cohomology  theory  defined  on proper  $G$-CW  complexes. Using a functorial  $G$-CW  approximation  for proper  $G$-ANR as introduced   in    \cite{antonyanelfving} for  locally  compact  Hausdorff  groups,  an  equivariant  cohomology  theory  may  be  extended  to  the  category  of  proper  $G$-ANR. 

More  generally,  the  \v{C}ech expansion  of \cite{matumotocech}   provides  a \v{C}ech   extension  of  a $G$-cohomology  theory to   arbitrary  pairs  of proper $G$-spaces. That  is, a  family  of  $R$-mod valued  functors  $\check{\mathcal{H}}_{G}^{n}$ defined  on pairs  of  proper $G$-spaces and  natural  transformations $\delta^{n}_{X,A}: \mathcal{H}^{n}_{G}(A, \emptyset)\to \mathcal{H}_{G}^{n+1}(X,A)$    satisfying  the  axioms: 

\begin{itemize}
\item{$G$-homotopy  invariance.}
\item{Long  exact  sequences  for $G$- pairs.}
\item{Excision.  Let  $X_{1}, X_{2}\subset X$  be  proper  $G$- invariant spaces such  that 
$$\overline{X_{2}-X_{1}}\cap X_{1}-X_{2}=\emptyset = X_{2}-X_{1} \cap\overline{X_{1}-X_{2} }$$ 
Then,    the inclusion  map $(X_{2},X_{1}\cap X_{2})\to (X_{1}\cup X_{2}, X_{1})$  induces  a  natural  isomorphism. }
\item{Axioms i-iii   for    the Induction  structure. }

\end{itemize}

\end{remark}

For the purposes  of   this work  we  need an  extension of  a  specific  cohomology  theory  to  a certain  proper $G$-ANR which  is  contractible after  forgetting  the  action   and  is exhausted   by  finite  $G$-CW  complexes.   This is  done  by  an  ad-hoc  construction, see definition \ref{definition adhoccohomotopy}.

Recall \cite{lueckdavis}, \cite{lueckeqcohomological}, that   for  any  Equivariant  Cohomology  Theory  $\mathcal{H}^{*}$ on   finite  $G$-CW  complexes    there  exists a spectral  sequence   with  $E^{2}$-term   given by  Bredon Cohomology
$$E_{2}^{p,q}=H_{\mathbb{Z} Or(G)}^{p}(X, \mathcal{H}^{-q}_{G}(G/H))$$
converging  to  $\mathcal{H}^{*}_{G}(X)$. 

The  following  result  will  be  used  later:

\begin{proposition}\label{proposition  spectral  sequence}
Let  $X$  be  an $l$-dimensional  $G$-CW  complex. Suppose that for  $r=2,3, \ldots$ the   differential appearing  in   the  Atiyah-Hirzebruch spectral  sequence  for  $X$ and $\mathcal{H}^{*}_{G}$  vanishes  rationally. Then,   for  any element  
$$x \in H^{0}_{\mathbb{Z} Or(G)} (X,\mathcal{H}^{0}_G(G/?))$$
 there  exists  some  positive  integer  $k$  such that  $x^{k}$ is contained  in  the  image   of  
 $\mathcal{H}^{0}_{G}(X)$ under  the edge  homomorphism 
$$ {\rm Edge}_{G}:  \mathcal{H}^{0}_{G}(X)\longrightarrow  H^{0}_{\mathbb{Z} Or(G)} (X,\mathcal{H}^{0}_{G}(G/?))$$ 
 \end{proposition}
\begin{proof}
Let  $x\in  H^{0}_{\mathbb{Z} Or(G)} (X,\mathcal{H}^{0}(G/?))$.  The  proof  reduces  to  construct  inductively  positive  integers $k_{2}, \ldots k_{l-1}$  such  that   the  product  $x^{\prod_{i=2}^{r} k_{i}}$  survives  to   $E^{0,0}_{r+1}$ for  $r=1,\ldots l-1$, in  the  sense  that  $k_{r}d_{r}^{0,0}(x^{\prod_{i=2}^{r-1}k_i})=0$ for  $r=2, \ldots, l-1$. Since  $x\in E_{2}^{0,0}$,   we  pick  $k_2$  such that  $k_2 d_2(x)=d_2(x^{k_2})=0$ (this  is  possible  by  the  rational  vanishing  of  the  differentials).

Assume  inductively  that  there  are $k_{2},\ldots, k_{r-1}$  and  $x^{\prod_{i=2}^{r-1} k_{i}}$  which  survive   to  the  $\in E_{r}^{0,0}$-term. 
Choose  $k_{r}$ such that $ k_{r}d_{r}^{0,0}(x^{\prod _{i=2}^{r-1}})=0$. This  is  possible  by  the  rational  vanishing  of  differentials  again). 
   
Now,  $d_{r}^{0,0}(x^{\prod_{i=2}^{r}})=k_{r}d_{r}^{0,0}((x^{\prod_{i=2}^{r-1}}))    (x^{\prod_{i=2}^ {r-1}} )^{k_{r}-1}$. 
And  since $x^{\prod_{i=2}^{r}} \in E_{r+1}^{0,0}$ for  $k=\prod_{i=2}^{l-1}k_{i}$,  the  $l$-dimensionality  of  $X$ implies  $x^{k}\in E_{\infty}^{0,0}$ and   hence it  is on the  image under  the  edge  homomorphism. 
\end{proof}

\begin{definition}(Universal Cohomology  Length relative  to  a family  of  subgroups)

Let  $\mathcal{A}=\{ G/H_{i}\}$  be  a  collection of  orbit  spaces representing  all  homogeneous $G$-spaces  with  isotropy  in some  family $\mathcal{F}$  of  subgroups of $G$.  Let  $M$  be  a  module  over  the graded   ring $ \mathcal{H}_{G}^{*}(\EGF{G}{\mathcal{F}})$.  The    $\mathcal{H}_{\mathcal{A}}$-length  of  the module  $M$  is  the  smallest  number  $k$  such  that   there  exist  spaces  $A_{1},\ldots,A_{k} \in  \mathcal{A}$  such that  for  any $\gamma \in M$  and $\omega_{i}$  in  the  kernel  of  the  map 

$$\mathcal{H}^{0}_{G}( \EGF{G}{\mathcal{F}})\to \mathcal{H}^{0}_{G}(G/H_{i})$$

given  by  the  up  to  $G$-equivariant  homotopy   unique  map $G/H\to \EGF{G}{\mathcal{F}}$,  one  has  
$$\gamma  \omega_{1} \ldots  \omega_k=0. $$ 

Given  a  map  $f:X\to  Y$,  between    $\mathcal{A}$-numerable  spaces, the  $\mathcal{H}_{\mathcal{A}}$-length  of   $f$  is  the  $\mathcal{H}_{\mathcal{A}}$ length  of  the   image, considered  as  $\mathcal{H}_{G}^*(\EGF{G}{\mathcal{F}})$- module   .   
\end{definition}

\section{Computations  in Burnside  Rings}
\typeout{-------------------------------- Section  2:  Computations  in Burnside  Rings  for  Infinite  Discrete  Groups-------------------------}

We  specialize  now  to  equivariant  stable  cohomotopy  for  proper  actions.

We  give  a  quick  summary  of  important  facts involving Equivariant  Stable  Cohomotopy  for  finite  groups.

\begin{theorem} \label{construction bartschelement}
Let  $G$ be  a  finite  group.  Then 
\begin{itemize}

\item{The  $0$-th   equivariant  cohomotopy group of a  point, $\pi_{G}^{0}(\pt)$ is  isomorphic  to  the  \emph{Burnside  ring},  denoted  by $A(G)$, the  Grothendieck  ring  of  isomorphism  classes  of   finite  $G$-sets .}
\item{ The  Burnside   ring  $A(G)$  is  provided  with   maps  $\varphi_{H}:A(G)\to \mathbb{Z}$, each  one  for  every  conjugacy  class  of  subgroups in $G$. These  extend to an  injective map  $A(G)\to  \underset{H{\rm  in \; ccs(G)}}  {\Pi} \mathbb{Z}$, where $ccs(G)$ denotes  the  set  of  conjugacy  classes  of  subgroups in $G$.}
\item{The  prime  ideals  in $A(G)$ are  given  by  the  sets $ \mathcal{P}_{K,p}=\{x \mid  \varphi_{H}(x)\cong 0 {(\rm p )} \}$, $\mathcal{P}_{H,0}=\{x\mid \varphi_{H}(x)= 0  \}$,  where  p  is  a  prime  number. The  augmentation ideal $I_{G}$ is  defined  as  the  ideal $\{x\mid \varphi_{e}(x)=0\}$. }  
\item{There  exists  an  element,   the  \emph{Bartsch element}  $ 0\neq x\in A(G)$  with the  property  that  $\varphi_{H}(x)=0$ for  every  subgroup $H$. }
\item{If   $p$ is  a prime  number  and  $G$ is  a finite p-group, then the completion  map  $A(G)\to  A(G)_{\hat{I_{G}}}$ is  injective and  the  $I_G$-adical  topology  and  the  $p$-adical topologies  coincide. }
\end{itemize}
\end{theorem}
\begin{proof}

\begin{itemize}

\item{This  is  well  known.  See \cite{segalicm}, \cite{tomdieck}.  }

\item{See \cite{tomdieck},  chapter II, section  8 , pages 155-160. The  image  is  characterized  by  a set of  congruences  for  the number of  generators of  cyclic   subgroups  of  the  \emph{Weyl groups} $NH/H$ for  every conjugacy  class of  subgroups $H$  in $G$  \cite{tomdieck}, section 5 chapter IV, page  256. Alternatively, Theorem 1.3  in \cite{laitinen}, page  41. }
\item{This  is  proven  in \cite{laitinen}, page  43, \cite{dress}. }
\item{This  is  done  in  \cite{mountainpassthmclapp}. The  element  is constructed  as  follows:  let $K$ be  a proper  subgroup of  $G$.  Put  $u_{K}=[G/K]- \mid G/K\mid^{K} [G/G]$. The  element $x$ is  defined  as  the  product  of  all  such  $u_{K}$, each  one  for  every  conjugacy class of  subgroups  in $K$.  }

\item{For  a  detailed  proof  see \cite{laitinen}. The  first  result, Corollary 1.11 in  \cite{laitinen},  follows  from  the  fact  that  in  this  situation the  kernel  of  the  completion  map,  $\cap_{n} I_{G}^{n}$  coincides  with  $\cap \ker(\varphi_{U})$,  where $U$ ranges  among  alll  $p$-sylow  groups.  The  second  result follows  from Frobenius  reciprocity  and  an  analysis  of  the  congruences defining  the  Burnside  ring  as  subring   inside $\underset{H{\rm  in \; ccs(G)}}  {\Pi} \mathbb{Z}$,  proposition 1.12  in \cite{laitinen}, page 44 .  }
\end{itemize}
\end{proof}

Equivariant  Cohomotopy for   proper  actions  of  infinite  discrete  groups  on  finite  $G$-CW complexes  was  defined  in  \cite{lueckeqstable}   via  finite  dimensional equivariant  vector  bundles for  proper, finite   $G$-CW  complexes. Alternative  approaches aer  given  by a  construction  using    nonlinear Fredholm cocycles, which  allow  actions  of  noncompact  Lie  groups on finite  $G$-CW complexes \cite{barcenasnonlinear}, as  well as  a  spectra  version  \cite{barcenasokayama}.   These  approaches  are  compared  in \cite{barcenastesis}. For  convenience,  we  give  the  definition  from  \cite{lueckeqstable}:

\begin{definition}
A $G$-vector bundle over a $G$-$CW$-complex $X$ consists of a real
vector bundle $\xi: E \to X$ together with a $G$-action on $E$ such
that $\xi$ is equivariant and each $g\in G$ acts on $E$ and $X$ via
vector bundle isomorphisms.\\
Let $S^\xi$ denote its fibrewise one-point compactification.
\end{definition}

\begin{definition} \label{definition equivariantcohomotopy}
Let $X$ be a proper $G$-$CW$-complex. Let $\SPHB^G(X)$ be the
category with
\begin{itemize}
\item $Ob(\SPHB^G(X)) = \{ $G$\text{-vector bundles over }X\}$; and
\item a morphism from a vector bundle $\xi: E\to X$ to vector bundle $\mu: F
\to X$ is given by a bundle map $u:S^\xi \to S^\mu$ which covers the identity
$id:X \to X$ and fiberwise preserves the basepoint.\\
(It is not required that $u$ is a fiberwise homotopy equivalence.)
\end{itemize}
\end{definition}

Let $\underline{\mathbb{R}^k}$ denote the trivial vector bundle $X
\times \mathbb{R}^k \to X$.
\begin{definition}
Fix $n\in \IZ$. Let $\xi_0, \xi_1$ be two $G$-vector bundles over
$X$, and let $k_0$ and $k_1$ be two non-negative integers such that
$k_i +n \ge 0$ for $i=0,1$. Then two morphisms
\[ u_i: S^{\xi_i \oplus \underline{\mathbb{R}^{k_{i}}}} \to S^{\xi_i \oplus
\underline{\mathbb{R}^{k_{i+n}}}}\] are called equivalent, if there
are objects $\mu_i$ in $\SPHB^G(X)$ for $i=0,1$ and isomorphisms of
$G$-vector bundles $v: \mu_0 \oplus \xi_0 \cong \mu_1\oplus \xi_1$
such that the following diagram in $\SPHB^G$ commutes up to homotopy
\end{definition}
\[
\xymatrix{ S^{\mu_0 \oplus \underline{\mathbb{R}^{k_1}}} \wedge_X
S^{\xi_0 \oplus \underline{\mathbb{R}^{k_{0}}}} \ar[r]^{id\wedge_X
u_0} \ar[d] & S^{\mu_0 \oplus \underline{\mathbb{R}^{k_1}}} \wedge_X
S^{\xi_0 \oplus \underline{\mathbb{R}^{k_{0}+n}}}\ar[d]
\\
S^{\mu_0 \oplus \xi_0\oplus \underline{\mathbb{R}^{k_{0}+k_{1}}}}
 \ar[d] &
S^{\mu_0 \oplus \xi_0\oplus \underline{\mathbb{R}^{k_{0}+k_{1}+n}}}
\ar[d]
\\
S^{\mu_1 \oplus \underline{\mathbb{R}^{k_{0}}}} \wedge_X S^{\xi_1
\oplus \underline{\mathbb{R}^{k_1}}} \ar[r]^{id\wedge_X u_1} &
S^{\mu_1 \oplus \underline{\mathbb{R}^{k_{0}}}} \wedge_X S^{\xi_1
\oplus \underline{\mathbb{R}^{k_{1}+n}}} }
\]

\begin{definition}
For a proper $G$-$CW$-complex $X$ define \[ \pi^n_G(X) =
\{ \text{equivalence classes of morphisms $u$ as above}\}
\]
\end{definition}
By  introducing   triviality  conditions  on a  $G$-CW  pair, (considering morphisms which  are  fibrewise constant with  the  value  the  point  at  infinity),  equivariant  cohomotopy  groups   are  extended  to  an equivariant  cohomology  theory  with  multiplicative  structure.

We  introduce  a  Burnside  ring  for  infinite  groups, making  out  of   Segal's  remark,    part  1  in Theorem \ref{construction bartschelement}, our definition  for  finite  groups:

\begin{definition}
Let $G$  be  a  group  with  a  finite   model  for  the  classifying  space   for  proper  actions $\eub(G)$.  The  Burnside  ring  for $G$ is  the $0$-th  equivariant  cohomotopy ring  of  the  classifying space  for  proper  actions.  In symbols 
$$A(G)= \pi_{G}^{0} (\eub(G))$$
\end{definition}

Denote  by $A^{\lim}(G)=\lim_{H\in \calfin}A(H) $  the  inverse  limit  of  the  Burnside  rings  of  the  finite  subgroups  of  $G$. Notice  that  this  agrees  with  the  $0,0$-entry  of  the  $E^{2}$-term  of  the  equivariant  Atiyah-Hirzebruch  spectral  sequence. The  following  relations  between  the  Burnside  ring  and  the   inverse-limit  Burnside  ring  are  easy   consequences  of the rational collapse  of the  Atiyah-Hirzebruch spectral  sequence:   

\begin{lemma}
Let  $G$   be  a  discrete  group  admitting  a finite model for the  classifying  space  for  proper  actions $\eub{G}$. 
   
\begin{enumerate}
 \item {The edge Homomorphism $e:A(G)\to  A^{\lim}(G) $  has  nilpotent  kernel  and  cokernel. Its  kernel  is  the   nilradical.}
 \item {The  edge  homomorphism  gives an  isomorphism   between the  set  of   prime  ideals  in $A(G)$ and  $A^{\lim}(G)$ (in  fact an homeomorphism  in  the  Zariski  topology),  by  assigning  a  prime  ideal $I\subset A^{\lim}(H)$ its  inverse  image $e^{-1}(I)\in A(G)$. }
 \item {The  rationalized  Burnside  ring  $\pi_{G}^{0}(\eub(G))\otimes \mathbb{Q}$ does  not  contain nilpotent  elements. }
\end{enumerate}
 
\end{lemma}
 
In the  rest  of  the  section  we will  describe  a completion  theorem for families  of p- groups  inside  finite  subgroups  of discrete  groups, which is  the  main  computational  tool for  the  computation  of  equivariant  cohomology  lengths needed  for  the proof  of  Theorem \ref{theorem mountainpass}.
 This  amounts  to a  generalization  of  the  Segal  Conjecture  for  families  \cite{haeberlyjackowskimay}.
 The  result  was  proved  in  \cite{barcenastesis}, Theorem 13 in page  58, although similar  results  have  been  proved in \cite{lueckoliverchern}, \cite{lueckolivervectorbundles} and  \cite{luecksegal}, from where  the  crucial  ideas  and methods come.

Let $G$ be  a  discrete group and $\mathcal{F}$ be  a  family  of
finite  subgroups  of $G$,  closed under  conjugation and  under
subgroups. Fix a  finite proper  $G$-CW complex  $X$ and  a  finite
dimensional  proper  $G$-CW  complex $Z$ whose  isotropy  subgroups lie  in $\mathcal{F}$. Let $f: X \to Z$ be  a  $G$-map. Regard $\pi_{G}^{0}(X)$
as a  module over $\pi_{G}^ {0}(Z)$. 

\begin{definition}
The  augmentation  ideal  with   respect  to   the  family $\mathcal{F}$  is  defined as  the  kernel  of  the  homomorphism
$$I=I_{G, \mathcal{F}, Z}= ( \pi_{G}^{0}(Z)\overset{{\rm
    res}_{G}^{H}\circ i^{*}}{\longrightarrow} \prod_{H \in\mathcal{F}}
    \pi_{H}^{0}(Z^ {0}) )$$

\end{definition}

\begin{proposition}
Let  $\mathcal{F}$ be a  family  of  finite $p$-subgroups. Assume  that  there  is   an  upper  bound  for  the order of  subgroups  in $\mathcal{F}$. 

Let $\mathcal{P}\subset \pi_{H}^{0}(\pt)$  be   a  prime  ideal. 

Then, the  ideal  $$I_{H, \mathcal{F}\cap H, \pt }:= \ker \pi_{H}^{0}(\pt)\to \prod_{K\in \mathcal{F}} \pi_{K\cap H}^{0}(\pt)$$ 
is  contained  in  $\mathcal{P}$  if   $\mathcal{P}$ contains  the   image  of  the  structure  map  for  $H$ 
$$\phi_{H}: \lim_{ K\in \mathcal{F}}  \pi_{K}^{0}(\pt) \to  \pi_{H}^{0}(\pt) $$
\end{proposition}

\begin{proof}

Let  $m$  be a positive  integer  number  divided  by  all  orders  of  subgroups  in $\mathcal{F}$. For  a  given  subgroup $K$ in the  family, Let  $u=\{u_{1}, \ldots, u_{m}\}$  be  a finite set of  cardinality $m$ with  a  free $K$-action. For  example, $u$  may  be  chosen  to  be  a   disjoint  union  of  $\frac{m}{\mid K\mid}$ copies  of $ K$. This  gives  an  injective  homomorphism  into  the  symmetric  group  in $m$  letters, $\rho: K\to S_{m}$. For  a  prime  $p$,  let  $Syl_{p}$  be  the  p-Sylow  subgroup of  $S_{m}$.  

 Let $S_{m}[\rho]$ be  the   set  $S_{m}$ with the free  $K$- action  given by  $k,s\mapsto \rho(h)(s)$ and  $S_{m}/Syl_{p}$    be  the  set  with the  induced  $K$-action.  Notice  that the fixed  point  set  $S_{m}/Syl_{p}^{L}$  is nonempty  if and  only  if  $L$ is a  p-subgroup.  This construction  is  compatible  with  morphisms between  subgroups  in $\mathcal{F}$ in the  sense  that  an  homomorphism $K\to K^{'}$  between  groups  in the family  induces  a  map  taking  the  free $ K^{'}$-set $ S_{m}$ to the   free $K$-set  $S_{m}$  and  the  same for  the   homogeneous  set  $S_{m}/Syl_{p}$. 

Consider  the elements 
$$\{( S_{m} - \mid S_{m} \mid K/K)\big {\}  } _{K \in\mathcal{F}}$$
  $$\{  (S_{m}/Syl_{p}- \mid S_{m}/Syl_{p}\mid K/K  \big{\}}_{K\in \mathcal{F}}$$

Let  $\mathcal{P}$ be  a  a prime  ideal  containing  the  image  of  the  structure  map  under  $\phi_{H}$.  By the  structure  of  the  prime  ideal  spectrum,  $\mathcal{P}$  is  of  the  form $\mathcal{P}(M, p)$, where $M$  is  a  subgroup  of  $H$  and  $p$  is  a  prime  number  or  zero.  By  assumption, $\mathcal{P}$ contains  the  image   under  the  structure  map of the  elements above. 
Since $\varphi^{M}(S_{m}-\mid S_{m}\mid)=\mid S_{m}\mid$ and  $\varphi^{M}(S_{m}/Syl_{p} -S_{m}/Syl_{p})=\mid S_{m}/Syl_{p}\mid ^{M} - \mid S_{m}/Syl_{p}\mid$  and  both  elements  belong to  $p\mathbb{Z}$, because  $S_{m}/Syl_p$  has  order prime  to  $p$,  we  conclude  that  either  $p=0$  or  $M$ is  a  $p$-group. 

If  $M$ is a  $p$-group,  then $\mathcal{P}(M, p)=\mathcal{P}(\{e\}, p)\supset \mathcal{P}(\{e\},0)\supset I_{\mathcal{F}, H, \pt}$.  If  $p$=0, then  $\mid S^{M}\mid -\mid S_{m}\mid=0$, and  hence  $M=\{e\}$. For any  subgroup $K^{'} $ of every  element $K\in \mathcal{F}\cap  H$, $\mathcal{P}(K^{'}, 0)=\mathcal{P}(\{e\},0)$,  since $K^{'}$  is  a  $p$-group, hence $P$ contains  the  intersection  of  all such  ideals, which  is  $I_{\mathcal{F}, H, \pt}$.

\end{proof}

\begin{proposition}\label{proposition primeidealssegal}
Let   $L$  be   an  $n$-dimensional  $G$-CW  complex with isotropy  in the  family   $\mathcal{F}$ consisting  of  finite  $p$-subgroups inside  the discrete  group $G$.  Let  $f:G/H\to L$  be a  $G$-map  and   $\mathcal{P}\subset \pi_{H}^{0}(\pt)$  be   a  prime  ideal. Then $I_{\mathcal{F}\cap H, \pt }:= \ker \pi_{H}^{0}(\pt)\to \prod_{K\in F} \pi_{K\cap H}^{0}(\pt)$ is  contained  in  $\mathcal{P}$  if   $\mathcal{P}$  contains  the   image   of  $I_{\mathcal{F}, Z}$ under  ${\rm  ind}_{H\to G}\circ f^{*}:\pi_{G}^{0}(L)\to  \pi_{H}^{0}(\pt)$
\end{proposition}

\begin{proof}
Let  $\mathcal{P}$  be  a prime  ideal containing $I_{\mathcal{F}, H, \pt}$. By  the previous proposition,  we  can  assume  that $\mathcal{P}$  contains  the  image  of  the  structural  map  $\phi_{H}$.

Let  $\psi: H^{0}_{\mathbb{Z}Or(G)}(\EGF{G}{\mathcal{F}}, \pi_{G}(G/?))\to \lim_{K}\pi_{K}^{0}(\pt)$  be the isomorphism  given  by  
assigning  to an  element  $x\in H^{0}_{\mathbb{Z}Or(G)}(\EGF{G}{\mathcal{F}};\pi_{K}^{0}(\pt)) $  the  element whose  component  under  the  structural  map $\phi_{K}$  is  the  image  image under the  map  induced  by  the ($G$-homotopically) unique map $u_{K}: G/K\to \EGF{G}{\mathcal{F}}$,  followed by  the induction  isomorphism 

\begin{multline*} 
$$ H^{0}_{\mathbb{Z}Or(G)}(\EGF{G}{\mathcal{F}}; \pi_{G}^{0}(G/?))\to \\  H^{0}_{\mathbb{Z}Or(G)}(G/K, \pi_{G}^{0}(G/?))\to   H^{0}_{\mathbb{Z}Or(K)}(\pt, \pi_{K}^{0}(K/?)) \cong\pi_{K}^{0}(\pt)   $$ 
 \end{multline*}
 Given  an element  $a\in \lim_{K}    I_{\mathcal{F}\cap K, \pt} $,  denote  by  $x$  its image  under $\psi^{-1}$. By proposition  \ref{proposition spectral sequence}, there  exist a  positive  integer $k$  and  an  element  $y\in \pi_{G}^{0}(\EGF{G}{\mathcal{F}})$  such  that  $edge(y)=x^{k}$,  which  is furthermore  an  element  of  $I_{\mathcal{F}, G, L}$. 

The  structure  map $\phi_{H}: \lim \pi_{K}^{0}(\pt) \to \pi_{H}^{0}(\pt)$ maps  $a^{k}$ to $\mathcal{P}$. Because $ \mathcal{P}$ is  a  prime  ideal, the  map ${\rm ind}\circ f^{*}$  maps $a  $  to $\mathcal{P}$. 
\end{proof}

\begin{theorem}[Segal  Conjecture  for  families  of  finite $p$-subgroups]\label{theorem segalconjecture}
Let $G$ be  a  discrete group and $\mathcal{F}$ be  a  family  of
subgroups  of  order $p$ of $G$ closed under  conjugation and  
subgroups. Fix a  finite proper  $G$-CW complex  $X$ and  a  finite
dimensional  proper  $G$-CW  complex $Z$ whose  isotropy  subgroups lie  in $\mathcal{F}$ and 
have  bounded  order. Let $f: X \to Z$ be  a  $G$-map. Regard $\pi_{G}^{0}(X)$
as a  module over $\pi_{G}^ {0}(Z)$ and  set 

$$I=I_{\mathcal{F}, Z}=\ker ( \pi_{G}^{0}(Z)\overset{{\rm
    res}_{G}^{H}\circ i^{*}}{\longrightarrow} \prod_{H \in\mathcal{F}}
    \pi_{H}^{0}(Z^ {0}) )$$
then 

$$\lambda_{X, \mathcal{F}, f}^{m}: \bigl\{ \pi_{G}^{m}(X) \diagup
  I^{n} \cdot \pi_{G}^{m}(X)   \bigr\}  \to   \bigl\{ \pi_{G}^{m}(
  E_{\mathcal{F}}(G)\times   X^{n-1}) \bigr\}$$ 
    
is  an isomorphism of pro-groups. Also, the  inverse  system  
$$\bigl \{ \pi_{G}^{m}((E_{\mathcal{F}}(G)\times X)^{n})\bigr\}_{n\geq
  1}$$
satisfies  the Mittag-leffler  condition. In particular
$${\rm lim}^{1} \pi_{G}^{m}((E_{\mathcal{F}}(G)\times X)^{n})=0$$
and $\lambda_{X, \mathcal{F},f}$ induces  an isomorphism 
$$\pi_{G}^{m}(X)_{I}\overset{\cong}{\longrightarrow}
\pi_{G}^{m}(E_{\mathcal{F}}(G)\times  X) \cong \lim_{n} \pi_{G}^{m}((E
_{\mathcal{F}}(G) \times X)^{n}) $$
\end{theorem}

\begin{proof}

Since  both  functors   have   Mayer-Vietoris  sequences,  both  of  the  systems  satisfy  the  Mittag-Leffler  condition  and  in  view  of  the 5-lemma  for pro-modules, \cite{atiyahmacdonald},  section 2,   an  inductive  argument  can  be  used  to  reduce  the  problem  to  the   situation of  $X=G/H$,  and  where  $H$  is  a  finite group.

In  this  case,  there exists   a  commutative  diagram 

 $$\xymatrix{\pi_{G}^{0}(Z) \ar[r]^{f^{*}} &
  \pi_{G}^{m}(G/H) \ar[d]_{{\rm ind}_{H\to
  G}^{\cong} } \\
 A(H) \ar[r]_{\cong} &  \pi_{H}^{0}(\{*\}) }  $$
 
Hence,  the  map   of  pro-modules $$\lambda_{X, \mathcal{F}, f}^{m}: \bigl\{ \pi_{G}^{m}(X) \diagup
  I^{n} \cdot \pi_{G}^{m}(X)   \bigr\}  \to   \bigl\{ \pi_{G}^{m}(
  E_{\mathcal{F}}(G)\times   X^{n-1}) \bigr\}$$ 
 factorizes  as follows  

$$\xymatrix { \bigl\{ \pi_{G}^{m}(G/H) \diagup
  I^{n} \cdot \pi_{G}^{m}(G/H)   \bigr\} \ar[r] \ar[d] &   \bigl\{ \pi_{H}^{m}(\pt) \diagup
  J^{n}   \bigr\} \ar[d] \\     \bigl\{ \pi_{G}^{m}(
  E_{\mathcal{F}}(G)\times   G/H^{n-1}) \bigr\}    &   \bigl \{  \pi_{H}^{0}(\pt)/ I^{n}_{\mathcal{F}\cap  H, H, \pt} \bigr \}  \ar[l]^{\cong} }$$

Where  $J$ is  the  ideal   generated  by  the  image of  $I$ under  ${\rm  ind}\circ f^{*}$ and  the  lower  horizontal  map  is  an  isomorphism  due  to  the  completion theorem  for  families  inside  finite  groups  of  \cite{haeberlyjackowskimay},  the  right  vertical map  is  induced  by $f$. Due  to   proposition \ref{proposition primeidealssegal}, the  prime  ideals containing  $J$ and  $I_{\mathcal{F}\cap H, h, \pt}$  agree  and  the  right  vertical  map  is an  isomorphism.

\end{proof}

\begin{corollary}\label{corollary p completion}
 Let  $p$  be  a prime  number. For  any  group  satisfying   conditions \ref{condition m}   for  which  the  maximal  finite  subgroups  are  finite  $p$-groups, 
 the  groups
$\pi_{G}^{0}(\eub{G})\otimes \mathbb{Z}_{\hat{p}}$ and  $\pi^{n}_{G}(\eub{G})_{\hat{\mathbb{I}}_{G,\mathcal{MAX}}}$  are  isomorphic.
\end{corollary}
\begin{proof}
The  morphism  of  pro-groups $\bigl \{\pi_{G}^{m}(X)/ p^{n} \pi_{G}^{n}(X) \bigr \} \to  \bigl \{ \pi_{G}^{m}(X\times  E\mathcal{MAX})^{n-1})\bigr \}$ is  proved   to  be  an  isomorphism  for $X= G/H$ with  $H$  a  p-group. 
The  prime  ideals  in $\pi_{H}^{0}(\pt)$  containing  $I_{\mathcal{MAX}\cap H, H, \pt }$  and the one    generated  by  the image  of $I_{\mathcal{MAX}, G, G/H}$ under  ${\rm  ind}\circ f^{*}$ agree  by  the  previous  argument. Because  $H$  is a  $p$-group, these  agree  with the ones  containing $I_{\mathcal{TR}, G, G/H}$  for  the  trivial  family.  Due  to  part 5 of  theorem \ref{construction bartschelement},  these  agree with the  ones  containing $p$. 

Since  both functors  have  Mayer-Vietoris  sequences, the  result  follows  by  induction  on the  dimension  of  $X$. 

\end{proof}

\begin{proposition}\label{proposition bartsch power}
 Let  $G$  be  a  discrete  group satisfying  conditions  \ref{condition m}. 
 
There  exists  a   `` Generalized  Bartsch  element '' $w \in \pi^{0}_{G}(\eub G)$   for  which  the  map   $\pi_{G}^{0}(\eub{G})\longrightarrow H^{0}_{\mathbb{Z}({\rm Or}(G)}(\eub{G}, \pi^{0}_{?}(\pt)= \lim_{K\in{\rm Sub(G)}}\pi_{K}^{0}(\pt)   \overset{\psi_{M} }{\longrightarrow} \pi_{M}^{0}(\pt)$ given  by  the  composition  of  the  edge  homomorphism  and  the  structural  map  for  the   inverse limit  maps  $w$  to  a power of  the  element 
  constructed  in  \ref{construction bartschelement}  for  any  maximal  subgroup $M$.

\end{proposition}

\begin{proof}

Let  $X_{M_{i}}\in \pi_{M_{i}}^{0}(\pt)$  be  the Bartsch  element  constructed  in  Theorem \ref{construction  bartschelement},  part 4. Put $x=\{ x_{M_{i}} \} \in \lim_{H} \pi_{H}^{0}(\pt)$. Choose    an  element $w$ and  a  power  $k$ such  that $w$   is  mapped  to   $x^k$ under  the  edge  homomorphism.  
\end{proof}

 \section{End  of  proof }

\begin{definition}\label{definition adhoccohomotopy}
Let  $\hat{X}$  be a  proper  and  paracompact $G$-ANR, which is  contractible  after  forgetting  the  group  action. Assume  that  there  is  a   map $X \to \hat{X}$ from  a   proper  $G$-CW  complex   of  finite  type $X=  \cup X_ {n}$ inducing a  weak  $G$-homotopy  equivalence (a  map  restricting  to   weak homotopy  equivalences  $\hat{X}^{H}\to X^ {H}$  for  all  subgroups $H$). 
Define 
  
$$\hat{\pi}_{G}^{*}(\hat{X})=  \lim_{n}\pi^{*}_{G}(X_{n})\otimes  \mathbb{Q}_{\hat{p}}$$ 
\end{definition}

\begin{proposition} \label{proposition powers}
Let  $G$ be  a discrete  group  satisfying  \ref{condition m}. Let $X$  be  a  paracompact  proper  $G$-ANR , which is  contractible  after  forgetting  the  group  action. Assume  that  there  is  a   map $X \to \hat{X}$ from  a   proper  $G$-CW  complex   of  finite  type $X=  \cup X_ {n}$ inducing a  weak  $G$-homotopy  equivalence.

The    maps  $X_{n} \to \eub{G}$  together with the  $G$-homotopy  equivalence $\eub{G}\to  J_{\mathcal{FIN}}(G) $ induce    isomorphisms

$$\hat{\pi}_{G}^{0}(J_{\mathcal{FIN}}(G))\to  \hat{\pi}_{G}^{0}(\eub{G})\overset{\cong}{\longrightarrow}\lim \pi_{G}^{0}(X_{n})$$ 

\end{proposition}

\begin{proof}
The  point  is  the  existence  of long  exact  sequences  for  the   functor $\hat{\pi}_{G}^{*}(X, A )$,  which  is  guaranteed  by  the natural  equivalence  with  the  Equivariant  Cohomology Theory   defined  by  $ (X,A) \mapsto \pi_{G}^{m}((\EGF{G}{\mathcal{MAX}}, \emptyset) \times ( X,A))$ on  finite  $G$-CW  pairs.

 \end{proof}

\begin{proposition} \label{proposition nonilpotents}
 Let  $G$ be  a  group  satisfying   conditions  \ref{condition m}. Let  $\hat{X}$  be a  proper $G$-ANR as  in  \ref{definition adhoccohomotopy}.  Then, there  exists  an  element  $w\in \pi_{G}^{0}(\eub{G})\otimes \mathbb{Q} $  such  that
\begin{itemize}

\item {$w\in \ker \pi_{G}^{0}(\eub{G})\otimes \mathbb{Q}  \to \pi_{G}^{0}(G/H)\otimes \mathbb{Q} $  for  all  finite  $H$. }
\item{$ w\in \ker \pi_{G}^{0}(\eub{G}) \otimes \mathbb{Q} \to   \pi_{G}^{0}(X_{0})\otimes \mathbb{Q}$. }
\item{For  every  $k>0$  there  exists  an $n>0$  such  that   the  image  if  $w^{k}$ under  $\hat{\pi}_{G}^{0}(\eub{G})\to \pi_{G}^{0}(X_{n})\otimes \mathbb{Q}_{\hat{p}} $ is  not  zero. }
\end{itemize}
  
\end{proposition}

\begin{proof}
Let  $v \in \pi_{G}^{0}(\eub{G})\otimes \mathbb{Q}\cong \Pi_{H\in \mathcal{MAX}}A(H)\otimes \mathbb{Q}$ be  the  image  of  the  element  constructed  in proposition  \ref{proposition  bartsch power}  under  the  rationalized  edge  homomorphism. 

Let  $m$ $=G$-$cat(X_{0})$  and  put  $w=v^{m}$. As  in  \cite{mountainpassthmclapp},  the  following  diagram  commutes:

$$\xymatrix{ \pi_{G}^{0}(\eub{G})\otimes \mathbb{Q}_{\hat{p}} \ar[d] \ar[rr] &  & \lim_{n} \pi_{G}^{0}(X_{n})\otimes \mathbb{Q}_{\hat{p}} \ar[d] \\  \pi_{G}^{0}(\eub{G})_{\hat{\mathbb{I}}_{G,\mathcal{MAX}}}\otimes \mathbb{Q}  \ar[r]^-{\cong} &  \hat{\pi}_{G}^{0}(\eub{G})\ar[r]^-{\cong}& \hat{\pi}_{G}^{0}(X) }$$

as  the  left  and  right vertical  maps  are isomorphisms,  and  there  are  no  nilpotent  elements  in  the rationalized Burnside  Ring $\pi_{G}^{0}(\eub{G})\otimes \mathbb{Q}$,   there  are  no  nilpotent elements  in  $\pi^{0}_{G}(\eub{G})\otimes \mathbb{Q}_{\hat{p}}$,  and  so  there  exists  a natural  number   $n$  such  that   the  third  condition  holds. 

\end{proof}

Let $\hat{E}\subset  E $  be  a  $G$-invariant  linear  subspace  with  a  finite  dimensional, $G$-invariant complement $F_{0}$  satisfying  the  mountain  pass  condition 1  in \ref{theorem mountainpass} . For any  finite  dimensional subspace  $\hat{F}$, the  sum  $F=F_{0}\oplus \hat{F}$ satisfies  
$$F-B_{r}(F)\subset \phi^{a}$$

\begin{lemma}
 There  is  a  $G$-map  $f$  such  that   the  diagram  
$$\xymatrix{ (F,F-B_{r}(F)) \ar[d]^{i_{F}} \ar[r] & (E-\{0\},\phi^{a}) \ar[d]^{f} \\ (F,F-S(F_{0}\oplus F)\ar[r]^{j_{F} } & (E-\{0\}, S(\hat{E}))} $$
\end{lemma}
commutes,  where  $i_{F}$ and  $j_{F}$  are  given  by  inclusions. 

\begin{proof}

Compare  lemma  5.2  in  \cite{clapppuppe}.  Define  a  map  $f:E\to \hat{E}$  by  sending   the   bounded  closed  subspace  $A$ in theorem \ref{theorem mountainpass} to  $0$,  mapping   $\hat{E}\cap \phi^{a}$ into  $\hat{E}-B_{r}(\hat{E})$ and  extending   to  all  of  $E$ , 
since  $\hat{E}$  is a  proper, $G$- absolute  retract,  Theorem  3.9 in  page  1953  of  \cite{sergeyleonardo}.
\end{proof}
   
The  same  argument  as  in   Proposition  5.3, \cite{clapppuppe},  page  17  yields: 

\begin{proposition}\label{proposition borsukulam}
  
 For  any  Equivariant  Cohomology  Theory, $\mathcal{H}_{G}^{*}$,  
$$G{\rm -}cat(E, \phi^{a})\geq \mathcal{H}^{*}_{G} {\rm lenght} \;(S(F_{0}\oplus  F) \to S(F_{0}\oplus  \hat{E}, S(F_{0}))$$
\end{proposition}

We  now  finish  the  proof  of  Theorem  \ref{theorem mountainpass}. This  follows  the  proof  of  proposition  3.2  in   
\cite{mountainpassthmclapp}. 

\begin{proposition}
$$G{\rm -}cat (E, \phi^{a})=\infty$$  
\end{proposition}

\begin{proof}
Let  $F_{n}$  be  an  increasing   sequence  of finite  dimensional  linear  $G$-subspaces of  $\hat{E}$  such  that  $\hat{F}=\cup F_{n}$  is  infinite  dimensional. 
as  in  \cite{mountainpassthmclapp},  the $\hat{\pi}^*_ {\mathcal{FIN}}$-length  of  the  inclusion

$$S(F_{0}\oplus F_{n})\to  S(F_{0}\oplus  \hat{E}, S(F_{0}))$$
becomes  arbitrarily  large  as  $n $ tends  to  infinity. 

The proper $G$-ANR  $S(\hat{E})$ satisfies  the  hypothesis  of  lemma  \ref{proposition  powers}.

Hence  there  is  an element  $w\in \pi_{G}^{0}(\eub{G})$  satisfying conditions  1  to  3  in  \ref{proposition powers}.  let  $v$ and  $v_{n}$  be  the   images  of  $W$  along  the  homomorphism induced  by  the   universal  maps $S(F_{0}\oplus \hat{E})\to  \eub{G}$, respectively $S(F_{0}\oplus \hat{F}_{n})\to  \eub{G})$. 
Since  the  diagram 

$$\xymatrix{\pi_{G}^{0}(S(F_{0}\oplus \hat{F}_{n})\ar[rr]^{j_{n}^{*}} & & \pi_{G}^{0}(S(F_{0}\oplus \hat{E}),S(F_{0})) \ar[ld]_{j^{*}} \\
& \pi_{G}^{0}(S(F_{0}\oplus \hat{E}) \ar[lu]  & }$$  
 commutes  up to homotopy, $v_{n}\in \rm{im}(j_{n}^{*})$, and  proposition  \ref{proposition powers} yields   that  for  any  $k$  there is  an  $n$   with  $\hat{\pi}_{\mathcal{FIN}}^{*}- {\rm lenght}  \;j_{n}\geq k$. 

\end{proof}

\section{Concluding  remarks }

Paraphrasing  Willem, \cite{Willemminimax}, page  3  Minimax-Type  Theorems  usually  consist  of  different   parts: 

\begin{itemize}
 \item  {Deformation lemma using  some  pseudo-gradient vector field. }
 \item  {Construction  of  Palais-Smale  typical  sequences, which  converge  either  due to some  \emph{a priori}  compactness  condition, or  which give  critical  points  using additional  \emph{a posteriori}  information, typically  \emph{topological intersection properties}, like  the  intermediate  value  theorem, the  Borsuk-Ulam theorem, degree notions, etc.}
\end{itemize}

In this  work,  the   proof  given  by  Bartsch-Clapp  Puppe  was  adapted  using  a   Borsuk-Ulam-Type Theorem,  which  may  be  deduced  from \ref{proposition borsukulam} and \ref{proposition nonilpotents}. The  problem  of  classifying  the  groups  satisfying equivariant Borsuk-Ulam-Type  theorems  has  deserved   particular  attention \cite{bartschborsukulam}, \cite{lahtinenmorimoto}, among  others.

Let  $G$  be  a  discrete, linear  group  which acts  properly and  linearly on  finite  dimensional  representation spheres $S^{V}$.  Define  the  Borsuk-Ulam  function  $b_{G}(n)$  as  the   maximal  natural  number $k$  such  that  if  there  exists  a $G$-map $S^{V}\to S^{W}$ where  $dim V\geq n$,  then  ${\rm dim} W \geq k$                                                                                                                                                                                                                                                                                                                                                                                                                                                                   
                                                                                                                                                                                                                                                                                                                                                                                                                                                             
\begin{problem}
Classify  all linear, discrete  groups  satisfying  
$$\lim _{n\to  \infty} b_{G}(n)=\infty$$ 
\end{problem}
 as  in \cite{bartschborsukulam}, \cite{lahtinenmorimoto},   and  in this   work, condition \ref{condition m},  the  answer  should  involve  restrictions  for  the    number of primes dividing  the cardinality  of  the  isotropy groups.

\begin{remark}[Topological Noncompact Groups  of Symmetry]
In   the  context  of  Hamiltonian  Systems,  some  proper  actions  of  non-compact Lie  groups  appear  \cite{robertswulfflamb}.  Equivariant  Cohomotopy  Theory   has  been  extended  in \cite{barcenastesis}, \cite{barcenasnonlinear} for  these   class  of  symmetries.  The use  of Equivariant  Algebraic  Topology, particularly  Equivariant  Cohomotopy  may  be  useful. However, in  this  context, the  Segal  Conjecture (which  was  the  main  homotopy theoretical  input  of  theorem \ref{theorem mountainpass}, crucially in  the   proof  of the  Borsuk-Ulam-type   result)  is  not  true, as  it  is   not  even  true   for   compact  Lie  groups, see \cite{feshbach}, \cite{bauer}.   
\end{remark}

\begin{remark}[Equivariant Degree Notions  for  Infinite  Discrete  Groups]
In  \cite{barcenastesis},  an equivariant  degree  notion  for   proper  actions  of discrete  group  is  defined. 
This  assigns  to  a  quadruple $(E,F,T,c)$  consisting  of   locally  trivial $G$- Hilbert  bundles over  a  proper, cocompact  $G$-CW  complex, a  fibrewise Fredholm operator $T$  and a  fibrewise compact  nonlinearity  satisfying  the  property that  the  map $T_{x}+c_{x}:E_{x}\to F_{x}$ defined  on   the  fibers $E_{x}, F_{x}$ over  each  point $x$  is  proper,  an  element  in  the equivariant cohomotopy $\pi_{G}^{*}(X)$, as introduced in  definition \ref{definition equivariantcohomotopy}. 
\end{remark}

\typeout{-------------------------------------- References  ---------------------------------------}

\bibliographystyle{abbrv}
\bibliography{mountainpass.bib}

\end{document}